\documentclass[12pt,reqno]{article}

\usepackage{amsmath,amstext, amsthm, amssymb}
\usepackage{a4wide}
\usepackage{color}
\usepackage{hyperref}
\usepackage[normalem]{ulem}

\usepackage[utf8]{inputenc}

\usepackage[T1]{fontenc} 
\usepackage[english]{babel} 

\numberwithin{equation}{section}

\newtheorem{theo}{Theorem}
\newtheorem{prop}{Proposition}

\newtheorem{coro}{Corollary}

\newtheorem{lem}{Lemma}
\newtheorem{defi}{Definition}
\newtheorem{questno*}{Question}

\theoremstyle{remark}
\newtheorem{Remark}{Remark}

\newtheorem*{Remarks*}{Remarks}

\def\Z{\mathbb{Z}}
\def\C{\mathbb{C}}
\def\N{\mathbb{N}}
\def\Q{\mathbb{Q}}

\def\n{\eta}

\def\n'{\nu}

\def\x{\chi}

\def\det{\mathrm{det}}

\def\Gal{\mathrm{Gal}}
\def\<{\langle}
\def\>{\rangle}

\def\GL{\mathrm{GL}}

\newcommand{\Qbar}{\overline{\mathbb Q}}

\newcommand{\K}{\mathbb K}

\begin{document}

\selectlanguage{english}

\title{Representability of $G$-functions as rational functions in hypergeometric series}
\date\today
\author{T. Dreyfus and T. Rivoal}
\maketitle

\begin{abstract}
Fres\'an and Jossen have given a negative answer to a question of Siegel about the representability of every $E$-function as a polynomial with algebraic coefficients in $E$-functions of type $_pF_q[\underline{a};\underline{b};\gamma x^{q-p+1}]$ with $q\ge p\ge 0$, $\gamma \in \Qbar$ and rational parameters $\underline{a}, \underline{b}$. In this paper, we study, in a more general context, a similar question for $G$-functions asked by Fischler and the second author: can every $G$-function be represented as a polynomial with algebraic coefficients in $G$-functions of type $\mu(x)\cdot _pF_{p-1}[\underline{a};\underline{b};\lambda(x)]$ with $p\ge 1$, rational parameters $\underline{a},\underline{b}$ and $\mu,\lambda$ algebraic over $\mathbb Q(x)$ with  $\lambda(0)=0$? They have shown the answer to be negative under a generalization of Grothendieck's Period Conjecture and a technical assumption on the~$\lambda$'s. 
Using differential Galois theory, we prove that, for every $N\in \mathbb N$, there exists a $G$-function which can not be represented as a  rational function with coefficients in $\overline{\C(x)}$ of solutions of linear differential equations with coefficients in $\mathbb C(x)$ and at most $N$ singularities in $\mathbb{P}^1 (\mathbb C)$. As a corollary, we deduce that not all $G$-functions can be represented as a rational function in hypergeometric series of the above mentioned type, when the $\lambda$'s are rational functions with degrees of their numerators and denominators bounded by an arbitrarily large fixed constant. This provides an unconditional negative answer to the question asked by Fischler and the second author for such~$\lambda$'s.
\end{abstract}

\section{Introduction}

Siegel \cite{siegel} defined in 1929 the notion of $E$-functions and $G$-functions as generalizations of the exponential and logarithmic functions respectively. We denote by $\Qbar \subset \mathbb C$ the field of algebraic numbers.
\begin{defi}\label{defi1}
A power series $F(x)=\sum_{n=0}^{\infty} \frac{a_n}{n!} x^n \in \Qbar[[x]]$, respectively $F(x)=\sum_{n=0}^{\infty} a_n x^n \in \Qbar[[x]]$, is an $E$-function, respectively a $G$-function, if 
\begin{enumerate}
\item[$(i)$] $F$ is solution of a non-zero linear differential equation with coefficients in 
$\Qbar(x)$.
\item[$(ii)$] There exists $C>0$ such that for any $\sigma\in \textup{Gal}(\Qbar/\mathbb Q)$ and any $n\ge 0$, $\vert \sigma(a_n)\vert \leq C^{n+1}$.
\item[$(iii)$] There exists $D>0$ and a sequence of integers $d_n$, with $1\le d_n \leq D^{n+1}$, such that
$d_na_m$ are algebraic integers for all~$m\le n$.
\end{enumerate}
\end{defi}
Note that $(i)$ implies that the $a_n$'s all lie in a certain number field $\K$.
A $G$-function has positive radius of convergence, finite unless it is a polynomial, and it can be analytically continued to a suitable cut plane (the origin of the cuts are amongst the singularities of the differential equation satisfied by the $G$-function); algebraic functions in $\Qbar[[x]]$ are $G$-functions.  
$G$-functions form a subring of $\Qbar[[x]]$ stable by differentiation and anti-differentiation (with algebraic integration constants). 

Classical examples of $E$- and $G$-functions are given by specializations of {\em the generalized hypergeometric function} 
\begin{equation} \label{eq:defhypfn}
{}_pF_{q}
\left[
\begin{matrix}
a_1, \ldots, a_p
\\
b_1, \ldots, b_q
\end{matrix}
\,; x \right] := \sum_{n=0}^\infty \frac{(a_1)_n\cdots (a_p)_n}{(1)_n(b_1)_n\cdots (b_q)_n} x^n
\end{equation}
where $p,q\ge 0$ and we define the Pochhammer symbol $(a)_n:=a(a+1)\cdots (a+n-1)$ for $n\ge 1$, $(a)_0:=1$. We shall also denote it by 
${}_pF_{q}[a_1, \ldots, a_p; b_1, \ldots, b_q; x]$. The parameters $a_j$ and $b_j$ are {\em a priori} in $\mathbb C$, with the restriction that $b_j\notin \mathbb Z_{\le 0}$ so that $(b_j)_n\neq 0$ for all $n\ge 0$. 
 We write a list $\underline{a}:=[a_1, \ldots, a_p]=\emptyset$ or $\underline{b}:=[b_1, \ldots, b_q]=\emptyset$ if $p=0$ or $q=0$, and the corresponding empty product of Pochhammer symbols is then set equal to 1. 
Let $\delta=x\frac{d}{dx}$. The hypergeometric series \eqref{eq:defhypfn} satisfies the {\em hypergeometric equation}
\begin{equation}\label{eq2}
\delta (\delta +b_1-1)\cdots (\delta +b_{q}-1)-x(\delta+a_1 )\cdots (\delta+a_p)y=0.
\end{equation}

 Siegel proved in 
\cite[\S9]{siegellivre} that, for any integers $q\ge p\ge 0$, the series 
\begin{equation}
\label{eq:Efnhyp}
{}_{p}F_q \left[ 
\begin{matrix}
a_1, \ldots, a_{p}
\\
b_1, \ldots, b_q
\end{matrix}
; x^{q-p+1}
\right] 
\end{equation}
is an $E$-function  when $a_j\in \mathbb Q$ and $b_j\in \mathbb Q\setminus \mathbb Z_{\le 0}$ for all $j$.  Similarly, for any integer $p\ge 1$, the series 
\begin{equation}
\label{eq:Gfnhyp}
{}_{p}F_{p-1} \left[ 
\begin{matrix}
a_1, \ldots, a_{p}
\\
b_1, \ldots, b_{p-1}
\end{matrix}
; x
\right] 
\end{equation}
is a $G$-function  when $a_j\in \mathbb Q$ and $b_j\in \mathbb Q\setminus \mathbb Z_{\le 0}$ for all $j$. More generally, any function of the form $\mu(x)\cdot {}_{p}F_{p-1}[a_1, \ldots, a_p; b_1, \ldots, b_{p-1}; \lambda(x)]$ is a $G$-function when $\lambda, \mu$ are algebraic functions in $\Qbar[[x]]$, with $\lambda(0)=0$. Note that in both cases, the rationality of the parameters is sufficient condition to obtain  $E$ and $G$-functions, but not a necessary one, as the following examples show (where $\alpha \in \Qbar^*)$:
\begin{equation}\label{eq:examplesnonrat}
{}_{2}F_1 \left[ 
\begin{matrix}
\alpha+1, 1
\\
\alpha
\end{matrix}
; x
\right] = \frac{\alpha(1-x)+x}{\alpha(1-x)^2}, \quad {}_{1}F_1 \left[ 
\begin{matrix}
\alpha+1
\\
\alpha
\end{matrix}
; x
\right] = \frac{x+\alpha}{\alpha}e^x.
\end{equation}
Galochkin has given in \cite{galochkin} necessary and sufficient conditions on the parameters of {\em non-terminating} hypergeometric series to be $E$- or $G$-functions; in particular these parameters must be algebraic.

\medskip

Siegel stated  \cite[p. 225]{siegel} a problem that we reformulate as the following question.
\begin{questno*}[Siegel] \label{questionE}
Is it possible to write any $E$-function as a polynomial with coefficients in $\Qbar$ of  $E$-functions of the form ${}_{p}F_q[a_1, \ldots, a_p; b_1,\ldots, b_q ;\gamma x^{q-p+1}]$, with $q\ge p\ge 0$, 
$a_j\in \mathbb Q$,  $b_j\in \mathbb Q\setminus \mathbb Z_{\le 0}$ and $\gamma\in \Qbar$?
\end{questno*}
The parameters $p,q, \gamma$ can take different values in the various hypergeometric series in the polynomial. This question has recently been answered in the negative by Fres\'an-Jossen \cite{frejos} who exhibited an explicit $E$-function which cannot be written as such a polynomial in hypergeometric series. A little earlier, a strong reason towards a negative answer had been given by  Fischler and the second author in \cite{firi3}. This strong reason is based on the incompatibility of a positive answer and a widely believed generalization to exponential periods of Grothendieck's Period Conjecture. In the same paper, the following question was also asked for $G$-functions, in the spirit of Siegel's question above.

\begin{questno*}[Fischler-Rivoal]\label{questionG}
Is it possible to write any $G$-function as a polynomial with coefficients in $\Qbar$ of $G$-functions of the form $\mu(x)\cdot{}_{p}F_{p-1}[a_1, \ldots, a_{p}; b_1,\ldots, b_{p-1} ;\lambda(x)]$, with $p \ge 1$, $a_j\in \mathbb Q$,  $b_j\in \mathbb Q\setminus \mathbb Z_{\le 0}$,   $\lambda, \mu\in \Qbar[[x]]$ algebraic over $\Qbar(x)$, and $\lambda(0)=0$?
\end{questno*} 
The parameter $p$ and the functions $\mu, \lambda$ can be different in the various hypergeometric series in the polynomial. It was proved in \cite{firi3} that a positive answer to Question~\ref{questionG}  again contradicts the above mentioned generalization of Grothendieck's Period Conjecture, but under the further assumption that the various functions $\lambda$ all tend to $\infty$ at a common singularity.\par  
 We observe also that  certain $\mathbb C(x)$-linear combinations of hypergeometric series with possibly transcendental parameters~(\footnote{hence not necessarily $G$-functions}) can represent $G$-functions; for instance, from the first identity in \eqref{eq:examplesnonrat}, we have  $\alpha( {}_2F_1[\alpha+1,1;\alpha;x]-{}_1F_0[1;\emptyset;x])=x/(1-x)^2$, where $\alpha$ can be any non-zero complex number.  For $\lambda\in \overline{\mathbb C (x)}$ and non-constant, we consider $_p\mathcal{H}_q[\underline{a};\underline{b};\lambda]$ the pullback of Eq.~\eqref{eq2} by the algebraic function $\lambda$; see for instance \cite[p.~4]{berkenbosch}. It is the monic linear differential equation over $\overline{\C(x)}$ with the $y(\lambda(x))$ as a $\mathbb C$-basis of solutions, where the $y(x)$ form a $\mathbb C$-basis of solutions of Eq.~\eqref{eq2}. When $\lambda\in \C$, we make the convention that $_p\mathcal{H}_q[\underline{a};\underline{b};\lambda]$  corresponds to $y'=0$.
We address the following more general problem.

\begin{questno*}\label{questionG'}
Is it possible to write any $G$-function as an element in the field $\mathcal{F}$ of  rational functions  with coefficients in $\overline{\C(x)}$
of solutions of  differential equations of the form 
$_p\mathcal{H}_{p-1}[\underline{a};\underline{b}; \lambda]$ with $p \ge 1$, $a_j\in \mathbb C$,  $b_j\in \mathbb C\setminus \mathbb Z_{\le 0}$, $\lambda\in \overline{\mathbb C (x)}$?
\end{questno*} 
The value of $p$ and the function $\lambda$ can be different in the various equations $_p\mathcal{H}_{p-1}[\underline{a};\underline{b};\lambda]$.

Since $1-x={}_{1}F_0[-1;\emptyset ;x]$, we have 
$1-{}_{1}F_0[-1;\emptyset ;\lambda (x)]=\lambda (x)$ for every $\lambda\in \overline{\mathbb C (x)}\setminus \C$. Therefore, in Question \ref{questionG'}, we obtain the same field $\mathcal{F}$ if we restrict the coefficients to be in $\mathbb C(x)$ or $\mathbb C$. However, this is no longer true if we restrict the functions $\lambda$ to be in $\mathbb C(x)$, as in the corollaries \ref{coro:N}, \ref{coro:N2}, \ref{coro:new} below obtained by taking coefficients in $\overline{\mathbb C (x)}$ (the statements would be weaker with coefficients in $\mathbb C(x)$ or $\mathbb C$). Note that in contrast with Question \ref{questionG}, we replace in Question \ref{questionG'} polynomials relations by rational relations, so that that many elements of $\mathcal{F}$ are not $G$-functions, not even order 0 arithmetic Nilsson-Gevrey series, defined and studied by Andr\'e \cite{andreannals}. Furthermore even when we replace {\em rational functions} by {\em polynomial}, thereby defining a ring $\widetilde{\mathcal{F}}$ in place of $\mathcal{F}$, certain elements of $\widetilde{\mathcal{F}}$ may not be $G$-functions because $\widetilde{\mathcal{F}}\not\subset \Qbar[[x]] $, and even  $\widetilde{\mathcal{F}}\not\subset \C[[x]] $.
 Question~\ref{questionG'} is of course interesting only for transcendental $G$-functions because $\overline{\mathbb Q (x)}\subset \mathcal{F}$.

\medskip

Our main result is a negative answer to a special case of Question \ref{questionG'} when the $\lambda$'s are non-constant rational functions with  numerators and denominators  of degree less than a fixed integer $N$.  The various differential equations may have non-common singularities so it is {\em a priori} not enough to construct a $G$-function with sufficiently many singularities. To solve this problem, we use an adaptation in our  context of Lemma 4.5 of \cite{frejos}, namely Proposition~\ref{prop3} below. We give a proof using Picard-Vessiot theory, but we emphasize that we follow the main ideas of the proof of \cite[Lemma 4.5]{frejos}, expressed in the Tannakian formalism. This result replaces the inclusion in a field of rational functions in several functions by the inclusion in a field of functions related by one single linear differential equation; see \S \ref{sec:gal} for the precise definition  of the notions of differential Galois theory used in this paper and \S \ref{sec:22} for the proof.

\begin{prop}\label{prop3}
Let $f$ be  a non-zero solution of a linear differential equation with coefficients in $\mathbb C(x)$ whose differential Galois group we denote by $G_f$
and with Picard-Vessiot extension $K_f$ containing $f$. Let $f_1,\dots, f_k$ be non-zero solutions  of linear differential equations with coefficients in $\C(x)$, with Picard-Vessiot extensions $K_{f_1},\dots,K_{f_k}$ containing $f_1, \ldots, f_k$ respectively.  Assume that $f\in \mathbb{C}(x)(f_1,\dots, f_k)\setminus \C(x)$.
If $G_f$ is  non commutative and has no normal algebraic subgroups other than itself and the trivial group, then there exists $i \in \{1,\dots, k\}$
such that $K_f\subset K_{f_i}$.
\end{prop}
Roughly speaking, $K_f \subset K_{f_i}$ means that $f\in \mathbb{C}(x)(f_{i,1},\dots ,f_{i,p_{i}})$, where $f_{i,1},\dots ,f_{i,p_{i}}$ form a maximal set of $\mathbb C$-linearly independent solutions of the  linear differential equation  satisfied by $f_i$.\par 

The elements of the Picard-Vessiot extension have the following property: there exists a finite set $S\subset \mathbb{P}^1 (\C)$ such that for all $x_0 \in \mathbb{P}^1 (\C)\setminus S$, there exists an open subset of $\mathbb{P}^1 (\C)$ containing $x_0$ and such that the elements can be analytically continued to that domain. The set $S$ is a subset of the set of singularities of the linear differential equation. By abuse of terminology, we shall say that elements of the Picard-Vessiot extension have a finite set of singularities. Since the elements of $\mathcal{F}$ are  rational functions  with coefficients in $\overline{\C(x)}$
of solutions of  differential equations of the form 
$_p\mathcal{H}_{p-1}[\underline{a};\underline{b}; \lambda]$, so that they are  rational functions  with coefficients in $\C(x)$
of solutions of  differential equations of the form 
$_p\mathcal{H}_{p-1}[\underline{a};\underline{b}; \lambda]$ and of algebraic functions.
On the other hand, the elements of the Picard-Vessiot extension of an algebraic function are algebraic. Therefore,  replacing the solutions of  $_p\mathcal{H}_{p-1}[\underline{a};\underline{b}; \lambda]$ by solutions of linear differential equation with at most $N\in \N$  singularities, to deduce Theorem~\ref{theo:N} from Proposition~\ref{prop3},  it is enough to display a transcendental $G$-function with a suitable differential Galois group and with  sufficiently many  non-polar singularities. 

\begin{theo}\label{theo:N} Let $N\in \N$. There exists a $G$-function $\xi_N$ which is not an element of the field $ \mathcal{K}_{N}$ of  rational functions with coefficients in $\overline{\C(x)}$ of solutions of linear differential equations with coefficients in $\mathbb C(x)$ and at most $N$ singularities in $\mathbb{P}^1 (\mathbb C)$.
\end{theo}
 As a corollary to Theorem~\ref{theo:N}, we obtain the following negative answer to a special case of Question~\ref{questionG'} where the $\lambda$'s are  rational functions with  numerator and denominator of degree less than any fixed integer~$M$. In this case,  the equations  $_p\mathcal{H}_{p-1}[\underline{a};\underline{b};\lambda]$ are of order at most $p$ with coefficients in $\mathbb C(x)$.

 \begin{coro}\label{coro:N}  Let $M\in \N^*$.  There exists a $G$-function
 which is not an element of the field   $\mathcal{F}_{M}$ of rational functions with coefficients in $\overline{\C (x)}$  
of solutions of  differential equations of the form 
 $_p\mathcal{H}_{p-1}[\underline{a};\underline{b};\lambda]$, with $p \ge 1$, $a_j\in \mathbb C$,  $b_j\in \mathbb C\setminus \mathbb Z_{\le 0}$, and $\lambda\in \mathbb C (x)$ with  numerators and denominators of degree less than $M$.
\end{coro}

 Specializing this corollary gives a negative answer to a special case of Question~\ref{questionG}. 

\begin{coro}\label{coro:N2}  Let $M\in \N^*$.  There exists a $G$-function which cannot be written as a polynomial with coefficients in $\Qbar$ of $G$-functions  ${\mu(x)\cdot{}_{p}F_{p-1}[a_1, \ldots, a_{p}; b_1,\ldots, b_{p-1} ;\lambda(x)]}$, with $p \ge 1$, $a_j\in \mathbb Q$,  $b_j\in \mathbb Q\setminus \mathbb Z_{\le 0}$,   $\lambda, \mu\in \Qbar[[x]]$, $\mu$ algebraic over $\Qbar(x)$, and $\lambda \in \overline{\Q}(x)\setminus  \overline{\Q}$ has  numerators and denominators of degree less than~$M$ 
and $\lambda(0)=0$.
\end{coro}
Again, it is understood that $p, \mu$ and $\lambda$ can vary in the polynomial.

\medskip

The $G$-functions $\xi_N$ in Theorem~\ref{theo:N} are all constructed using the $G$-function $\xi(x):=x(x^2-34x+1)^{1/2}\alpha(x)$, where $\alpha$ is the generating $G$-function of Ap\'ery's numbers:
$$
\alpha (x):=\sum_{n=0}^\infty \left(\sum_{k=0}^n \binom{n}{k}^2\binom{n+k}{n}^2\right)x^n
\in \mathbb Z[[x]],
$$
which is a solution of 
$$
 x^2(1-34x+x^2)y'''+x(3-153x+6x^2)y''
+
(1-112x+7x^2)y'+(x-5)y=0. 
$$
For our purpose, the crucial properties of $\xi$  is that 
it is transcendental
 and the differential Galois group of one linear differential equation of order 3 it satisfies is $\mathrm{PSL}_2 (\C)$. 

\medskip

 As another direct application of  Proposition~\ref{prop3}, we shall also prove that the function $\xi$ provides a negative answer to the question for $G$-functions which is the closest to Siegel's original question for $E$-functions. 
\begin{coro} \label{coro:new}
 The $G$-function $\xi$ cannot be written as a polynomial with coefficients in $\Qbar$ of functions of the form ${{}_{p}F_{p-1}[a_1, \ldots, a_{p}; b_1,\ldots, b_{p-1} ;\gamma x^m]}$, with $p \ge 1$, $a_j\in \mathbb Q$,  $b_j\in \mathbb Q\setminus \mathbb Z_{\le 0}$, $\gamma \in \overline{\Q}^*$, $m\in \Z^*$. 
\end{coro}
It is understood that the parameters of the hypergeometric functions can vary in the polynomial, as well as the values of $\gamma$, $p$ and $m$.

\medskip

We conclude this introduction by mentioning that the question of the relations between $G$-functions of differential order 2 and hypergeometric series $_2F_1$ was the subject of a conjecture of Dwork \cite{dwork2}, disproved by Krammer \cite{krammer}. See the introduction of \cite{firi2} for more details, as well as some results on the nature of $G$-functions of differential order $\le 2$. On a similar note, it is shown in \cite{hsz} that a certain (explicit) period of a family hyperelliptic curves of genus 2 (related the $G$-function generating the squares of Legendre polynomials) satisfies a linear differential equation of order 2 with monodromy group dense in $\mathrm{SL}_2(\mathbb R)$, amongst other examples, from which the authors conclude that ``this suggests that they (these equations) cannot be solved in terms of $_2F_1$ hypergeometric functions, which is novel for an arithmetic second order equation that is defined over $\mathbb Q$''. These results have no obvious intersection with our results, which are of a different nature. Let us also 
mention the assertion that belongs to folklore that $G$-functions could all be obtained as suitable ``specializations'' of the parameters and variables of multivariate $A$-hypergeometric functions (defined in \cite{gkz}). This is already
known to be true for algebraic functions over $\mathbb Q(x)$, see \cite{sturmfels} for instance.

\medskip

The paper is organized as follows. Section \ref{sec1} is devoted to a quick reminder of differential Galois theory, and we prove there Proposition \ref{prop3}. In Section \ref{sec:specialGfn}, we study the $G$-function~$\xi$ and compute its differential Galois group. Finally in Section~\ref{secfin}, we give the proofs of Theorem \ref{theo:N} and 
Corolla\-ries~\ref{coro:N} and~\ref{coro:new}.

\bigskip

\noindent {\bf Acknowledgement}. We thank M. Mezzarobba for his explanations concerning his Maple code {\tt NumGfun} and A. Bostan for his comments on Schwarz' classification. Both authors have partially been funded by the ANR project De Rerum Natura (ANR-19-CE40-0018). The IMB receives support from the EIPHI Graduate School (contract ANR-17-EURE-0002).  We thank the referee for a careful reading that helped us to improve the first version of this paper.

\section{A result in differential Galois theory}\label{sec1}

This section is devoted to differential Galois theory. It contains a short survey and  the proof of Proposition \ref{prop3}. We also state and prove Lemma \ref{prop4}, which is necessary for the proof of Theorem \ref{theo:N}.

\subsection{A quick survey of differential Galois theory}\label{sec:gal} The goal of this section is to make a quick remainder of differential Galois theory.  For convenience, we shall now focus on the case where the equation has coefficients in $\mathbb{C}(x)$.  For a general theory, we refer to \cite{putsinger}.   
Let $B$ be a $p\times p$ square matrix with entries in $\mathbb{C}(x)$ and 
let us consider the linear differential system 
$Y' =BY.$ 
In what follows, for $K$ a differential field extension of $\mathbb{C}(x)$, we set ${K^{\partial}:=\{ a\in K|a'=0\}}$ its field of constants. Note that $\mathbb{C}(x)^{\partial}=\mathbb{C}$ has characteristic zero and is algebraically closed. 
A Picard-Vessiot extension for $Y' =BY$ over $\mathbb{C}(x)$ is a differential field extension $K/\mathbb{C}(x)$ such that 
\begin{enumerate}
\item[--] There exists $U\in  \mathrm{GL}_p (K)$ such that $U'=BU$; such a matrix $U$ is called a fundamental solution;
\item[--] $K$ is generated over $\mathbb{C}(x)$ by the entries of $U$;
\item[--] $K^{\partial}=\mathbb{C}$.
\end{enumerate}
By the Cauchy-Lipschitz theorem, for all $a\in \C$ that is not a pole of any entry of $B$,  there exists a fundamental solution $U_a$ with entries analytic at $a$. Then, 
$\C(x)(U_a)$ has field of constants $\C$ and  is therefore a Picard-Vessiot extension for $Y'=BY$ over $\C(x)$, proving the existence of the latter.  Then  we can realize the Picard-Vessiot extension as a field of meromorphic functions at a  point that is not a singularity of the equation. A Picard-Vessiot extension is unique up to isomorphisms of  $\mathbb{C}(x)$-differential algebras.  Moreover, given a vector of meromorphic solutions, it is possible to construct a Picard-Vessiot extention for $Y'=BY$ over $\C(x)$ that contains this vector.  If the vector is non-zero, we can even impose that the first column of the fundamental matrix $U$ is this vector solution.  

\begin{Remark}
A linear differential equation 
\begin{equation}\label{eq1}
a_p y^{(p)}+a_{p-1} y^{(p-1)}+\dots +a_0 y=0,  \qquad a_i\in \C[x],\quad a_p\neq 0,
\end{equation}
is equivalent to the differential system 
\begin{equation}\label{sys1}
Y'= \begin{pmatrix}
0&1&0&\cdots&0\\
0&0&1&\ddots&\vdots\\
\vdots&\vdots&\ddots&\ddots&0\\
0&0&\cdots&0&1\\
-a_{0}/a_p& -a_{1}/a_p&\cdots & \cdots & -a_{p-1}/a_p
\end{pmatrix} Y.
\end{equation}
By abuse of terminology,  we define the Picard-Vessiot extension for \eqref{eq1} over  $\C(x)$ as the Picard-Vessiot extension for \eqref{sys1} over $\C(x)$. The  singularities of the differential equation in $\C$ are the zeros of $a_p$ (when the entries are coprime).
\end{Remark}

We define  the differential Galois group of $Y' =BY$ over $\C(x)$ as the group of field automorphisms of $K$ that leave invariant $\C(x)$, and that commute with the derivation.  Let us denote it by $\mathrm{Gal}(K/\C(x))$.  
Since $K^{\partial}=\C$,
we have a faithfull representation,  
\begin{eqnarray*}
\Gal(K/\C(x)) & \rightarrow & \GL_{p}(\C) \\ 
  \sigma & \mapsto & U^{-1}\sigma(U).
\end{eqnarray*}
The latter representation  identifies $\Gal(K/\C(x)) $ with a linear algebraic subgroup $G\subset\GL_{p}(\C)$. The bigger this group is, the fewer algebraic relations there are in $K$.  This representation depends upon the choice of the fundamental solution $U$. Choosing another solution will lead to a conjugated representation.  Let us denote by $G$ the identification of $\Gal(K/\C(x))$ as an algebraic subgroup of $\GL_{p}(\C)$.  We define the   differential Galois group of \eqref{eq1} over $\C(x)$  as the differential Galois group of   \eqref{sys1} over $\C(x)$.
\par 
Let us now describe  the Galois correspondence that will be used in the sequel (see \cite[Proposition~1.34]{putsinger}). We use the above notations.
\begin{prop}\label{prop1}
Let $\mathcal{G}$ be the set of algebraic subgroups of the differential Galois group $G$ and let $\mathcal{F}$ be the set of differential subfields of $K$ containing $\C(x)$.  For $H\in \mathcal{G}$, let $K^H\in \mathcal{F}$ be the differential subfield of $K$ whose elements are invariant under the action of $H$. For $F\in \mathcal{F}$, let $\mathrm{Gal}(K/F)$  be the differential field automorphisms of $K$ leaving $F$ invariant. Then, 
the following holds. 
\begin{enumerate}
\item[--] The map $ H\mapsto  K^H$ defines a bijection between  $\mathcal{G}$  and $\mathcal{F}$. Its inverse is given by $F\mapsto \mathrm{Gal}(K/F)$. 
\item[--] Let $H\in \mathcal{G}$.  Then, $H$ is a normal subgroup of $G$ if and only if $F:=K^H$ is stable under the action of $G$. In that case  the map $G\rightarrow \mathrm{Gal}(F/\C(x))$ is surjective with kernel $H$.
\end{enumerate}
\end{prop}

In the sequel we shall use the following definition and remark. 

\begin{Remark}\label{rem1}
When $G$ has the property that it has no normal algebraic subgroup other than itself and the identity,  it follows from the Galois correspondence of Proposition \ref{prop1} that the only subfields of $K$  stable under the action of $G$ are $K^G=\C(x)$ and $K^{\mathrm{Id}}=K$.  This fact will be a crucial part of the proof of Proposition \ref{prop3}. Note that such a property for $G$ is related to, but slightly differs from, the definition of {\em simple algebraic groups} in \cite[p.~168]{humphreys2012linear}. Recall that an algebraic group over a field is said to be simple if it is non-commutative and has no  connected  normal algebraic subgroups other than itself and the identity. For instance $\mathrm{SL}_2 (\C)$ is simple in this algebraic sense but does not satisfies the above assumption, since $\{\pm \mathrm{Id}\}$ is a (non connected) normal algebraic subgroup. On the other hand it is a classical result that the only proper  normal subgroups of $\mathrm{SL}_2(\C)$ consist of scalar matrices; see for instance the introduction of \cite{costa1988normal}. Hence, the only proper normal subgroup of $\mathrm{SL}_2(\C)$
is $\{ \pm \mathrm{Id}\}$. Since $\mathrm{PSL}_2 (\C)$ is isomorphic to $\mathrm{SL}_2(\C)/\{ \pm \mathrm{Id}\}$, it is simple as an abstract group and therefore 
 satisfies the assumption at the beginning of this remark.
\end{Remark}

The following technical result will be used in the proof of Theorem \ref{theo:N}.

\begin{lem}\label{prop4}
Let $Y'=AY$ a linear differential system with coefficients in $\C(x)$. Let $U$ be a fundamental solution with meromorphic functions entries.  Let $K=\C(x)(U)$ be the Picard-Vessiot extension and $G$ be the differential Galois group associated to this system. 

Let $\varphi\in \C(x)$ be non-constant. Then $U(\varphi(x))$ is a fundamental solution of $Y'=\varphi'(x)A(\varphi(x))Y$. The field $K_{\varphi}=\C(x)(U(\varphi(x)))$ is a Picard-Vessiot extension for $Y'=\varphi'(x)A(\varphi(x))Y$ over $\C(x)$ and the differential Galois group $G_{\varphi}$ over $\C(x)$ has the same dimension as $G$.
\end{lem}

\begin{proof}
The field $K_{\varphi}=\C(x)(U(\varphi(x)))$ consists in meromorphic function on a certain domain. Its field of constants is $\C$ and therefore $K_{\varphi}$ is a Picard-Vessiot extension for $Y'=\varphi'(x)A(\varphi(x))Y$ over $\C(x)$. Every algebraic relation amongst elements of $K=\C(x)(U(x))$ corresponds to an algebraic relation amongst elements of $K_{\varphi}=\C(x)(U(\varphi(x)))$ by replacing $x$ with $\varphi(x)$. Since the inverse of $\varphi$ is algebraic, every algebraic relation amongst elements of $K_{\varphi}=\C(x)(U(\varphi(x)))$ corresponds to an algebraic relation amongst elements of $\overline{\C(x)}(U(x))$. This shows that $\overline{\C(x)}(U(x))$ and  $\overline{\C(x)}(U(\varphi(x)))$ have the same transcendence degree over $\overline{\C(x)}$. Hence $\C(x)(U(x))$ and $\C(x)(U(\varphi(x)))$ also have the same transcendence degree over $\C(x)$.~By \cite[Corollary~1.30]{putsinger}, the differential Galois groups of $Y'=A(x)Y$ and $Y'=\varphi'(x)A(\varphi(x))Y$  have the same dimension.
\end{proof}

\subsection{Proof of Proposition \ref{prop3}}\label{sec:22}

Let us prove Proposition \ref{prop3} stated in the Introduction.
The linear differential equation satisfied by $f$ is equivalent to a linear differential system $Y'=A_f Y$.
Let $U_f$ be a fundamental solution such that $K_f=\C(x)(U_f)$. Similarly, consider $Y'=A_{f_i}Y$ the differential system satisfied by $U_{f_i}$  and let $K_{f_i}=C(x)(U_{f_i})$.  Let us consider the linear differential system  $Y'=\mathrm{Diag}(A_f,A_{f_{1}},\dots,A_{f_{k}} )Y$. Because the entries of the matrices of the differential systems are in the field $\C(x)$, we may assume that the functions in the various Picard-Vessiot extensions all have a common domain of meromorphy.
A fundamental matrix solution of this system can be taken to be $\mathrm{Diag}(U_f,U_{f_{1}},\dots, U_{f_k})$ and  the field $L:=\C(x)(U_f,U_{f_{1}},\dots,U_{f_{k}})$ consists in meromorphic functions, which enables us to embed each of the fields $K_f,K_{f_{1}},\dots, K_{f_{k}}$ into the larger field $L$. Note that since $L$ consists in meromorphic functions, its field of constants is $\C$ and therefore $L$ is a Picard-Vessiot extension for $Y'=\mathrm{Diag}(A_f,A_{f_{1}},\dots,A_{f_{k}} )Y$ over $\C(x)$.\par 

 Up to replacing  $\{K_{f_{1}},\dots,K_{f_{k}}\}$ by a smaller set, we may reduce to the case where  a relation of the form $f\in \mathbb{C}(x)(g_1,\dots, g_k)$ with $g_i\in \{K_{f_{1}},\dots,K_{f_{k}}\}$ should involve each Picard-Vessiot extension $K_{f_i}$. The case $k=1$ is trivial. So let us assume that $k\geq 2$.\par 
Consider  the  system  $Y'=A_{2}Y$,  where $A_{2} =\mathrm{Diag}(A_{f_i})$ for all $2 \leq i \leq k$ such that $f_i\notin K_{f_1}$.  We define the Picard-Vessiot extension $K_{2}$ for $Y'(x)=A_{2} Y$ over the field $\C(x)$ as $\C(x)(U_{2})$, where $U_{2}:=\mathrm{Diag}(U_{f_i})$ for $f_i\notin K_{f_1}$. Let $G_{2}$ be the differential Galois group of the differential system $Y'=A_2Y$. \par 
Let us consider the linear differential system  $Y'=\mathrm{Diag}(A_f,A_{f_{1}},A_{2})Y$.
A fundamental matrix solution of this system can be taken to be $\mathrm{Diag}(U_f,U_{f_{1}},U_{2})$. By construction, we have that $\C(x)(U_f,U_{f_{1}},U_{2})\subset L$ is a Picard-Vessiot extension for $Y' =\mathrm{Diag}(A_f,A_{f_{1}},A_{2})Y$ over $\C(x)$; we denote by $G_{\mathrm{diag}}$ the corresponding differential Galois group. \\  \par 
 We shall prove the theorem as follows: assuming that $f\in \mathbb{C}(x)(f_1,\dots, f_k)$ is such that $K_f \not \subset  K_{f_i}$ for all $i$, and also that $G_f$ satisfies the assumptions of Proposition \ref{prop3}, then we shall conclude that $f\in \C(x)$.
By assumption, 
$K_f \not \subset K_{f_{1}}$. By the minimal property of $\{K_{f_{1}},\dots,K_{f_{k}}\}$, we also have $K_f \not \subset  K_{2} $. 

\medskip

The elements of $G_{\mathrm{diag}}$ are of the block diagonal  form $\mathrm{Diag}(C_f, C_{f_1},C_{2})$, where $C_f\in G_f$, $C_{f_1}\in G_{f_1}$,  $C_2\in G_2$.
Then, $G_{\mathrm{diag}}$ is contained inside $G_f \times G_{f_1}\times G_{2}$. It is important to note that $\C(x)(U_f)$  is stable under the image of $G_{\mathrm{diag}}$, so that  by Proposition \ref{prop1}, the projection map  $G_{\mathrm{diag}}\rightarrow G_f$  defined by $\mathrm{Diag}(C_f, C_{f_1},C_2)\mapsto C_f$ is surjective. This shows that any element of $G_f$ can be extended as an element of $G_{\mathrm{diag}}$. 
\par 
By  Proposition~\ref{prop1}, there exists an algebraic subgroup $\widetilde{H} \subset G_{\mathrm{diag}}$ such that $L^{\widetilde{H}}=\C(x)(U_{f_1})$. By definition, the action of $\widetilde{H}$ on the fundamental solution $\mathrm{Diag}(U_f,U_{f_1},U_2)$ leaves the  block $U_{f_1}$ invariant. So the elements of $\widetilde{H}$ are of the form $\mathrm{Diag}(C_f,\mathrm{Id},C_2)$, where $C_f\in G_f$, $C_2\in G_2$. 
Moreover by  Proposition~\ref{prop1}, since $\C(x)(U_{f_1})$  is stable under the image of $G_{\mathrm{diag}}$, $\widetilde{H}$ is  a normal subgroup of $G_{\mathrm{diag}}$.  Additionally, the elements $C_f$  in  $\mathrm{Diag}(C_f, \mathrm{Id} ,C_2)\in \widetilde{H}$ generate a subgroup $H \subset G_f$. 

Let us prove that $H$ is a normal subgroup of $G_f$.
Since $\widetilde{H}$ is a normal subgroup of $G_{\mathrm{diag}}$, we have for all $g=\mathrm{Diag}(C_f, C_{f_1},C_2)\in G_{\mathrm{diag}}$, $g\widetilde{H}=\widetilde{H}g$. Equating the first block and using the fact that the projection map $G_{\mathrm{diag}}\rightarrow G_f$ is surjective, 
we find $C_f H=HC_f$ for all $C_f \in G_f$, proving that $H$ is a normal subgroup of  the group $G_f$.  By assumption on the group $G_f$, there are two possibilities: either $H=G_f$ or $H=\{ \mathrm{Id} \}$. 

$\bullet$ If $H=\{\mathrm{Id}\}$, then the elements of $\widetilde{H}$ are of the form $\mathrm{Diag}(\mathrm{Id},\mathrm{Id},C_2)$. The action of $\widetilde{H}$ on the fundamental solution $\mathrm{Diag}(U_f,U_{f_1},U_2)$ then leaves the first block $U_f$ invariant. Since $K_f=\C(x)(U_f)$ this implies that $\widetilde{H}$ also  leaves $K_f$ invariant. Then $K_f \subset L^{\widetilde{H}}$. Since $L^{\widetilde{H}}=K_{f_1}$, this leads to the conclusion that we have $K_f \subset K_{f_1}$.  A contradiction. \par 
$\bullet$ Therefore,  $H=G_f$. This means that for all $C_f\in G_f$, there exists $C_2\in G_2$ such that  $\mathrm{Diag}(C_f,\mathrm{Id},C_2)\in G_{\mathrm{diag}}$. Similarly, taking the algebraic subgroup $\widetilde{H}' \subset G_{\mathrm{diag}}$ such that $L^{\widetilde{H}'}=\C(x)(U_2)$ we construct a normal subgroup $H'\subset G_f$ similar to $H$ that is either $\{\mathrm{Id}\}$  or $G_f$. In the first case, we deduce that  $K_f \subset K_2$ (a contradiction). Then, $H'=G_f$ and   
for all $C_f'\in G_f$, there exists $C_{f_1}\in G_{f_1}$ such that  $\mathrm{Diag}(C_f',C_{f_1},\mathrm{Id})\in G_{\mathrm{diag}}$. 
 \par 
Let $C_f,C_f'\in G$ and consider
$\mathcal{A}:=\mathrm{Diag}(C_f,\mathrm{Id},C_2),\mathcal{B}:=\mathrm{Diag}(C_f',C_{f_1},\mathrm{Id})\in G_{\mathrm{diag}}$. A calculation shows
$M:=\mathcal{A}\mathcal{B}=\mathrm{Diag}(C_f C_f',C_{f_1},C_2)$ and $M':=\mathcal{B}\mathcal{A}=\mathrm{Diag}(C_f'C_f ,C_{f_1},C_2).$ Then $M M'^{-1}=\mathrm{Diag}(C_f C_f'(C_f'C_f)^{-1},\mathrm{Id},\mathrm{Id})$. We claim that the latter matrix is the identity. 
 Let $\psi \in G_{\mathrm{diag}}$ that corresponds to $M M'^{-1}$ and let us prove  that for all $\sigma\in G_{\mathrm{diag}}$, $\psi$ fixes $\sigma(f)$. 
  Let $\sigma \in G_{\mathrm{diag}}$.   Since $f=P(f_1,\dots, f_k)$, with $P\in \C(x)(X_1,\dots,X_k)$, we have $\sigma(f)=P(\sigma(f_1),\dots,\sigma(f_k))$ with $\sigma (f_i)\in \{K_{f_1},K_2\}$. 
  Since $\psi$ induces the identity on $K_{f_1}$ and $K_2$, it follows that  $\psi(\sigma (f_i))=\sigma(f_i)$. Hence, 
  $$
  \psi\circ \sigma (f)=P(\psi\circ \sigma(f_1),\dots,\psi\circ\sigma(f_k))=P(\sigma(f_1),\dots,\sigma(f_k))=\sigma(f),
  $$ proving that $\psi$ fixes $\sigma(f)$.  
  Let $\widetilde{K}_f\subset K_f $ be the smallest differential field containing $\C(x)(f)$, such that $\sigma( \widetilde{K}_f)\subset \widetilde{K}_f$, for all $\sigma\in G_f$. 
  Given our assumption on the algebraic normal subgroups of $G_f$, we can now use Remark \ref{rem1} and deduce that either $\widetilde{K}_f=\C(x)$ or $\widetilde{K}_f=K_f$.
    Consider the case $\widetilde{K}_f=K_f$. From what precedes, for all $g\in \widetilde{K}_f$, we have  $\psi(g)=g$. This shows that $\psi$ is the identity on $\widetilde{K}_f=K_f=\C(x)(U_f)$. Then $\psi$ leaves the block $U_f$ invariant in $\mathrm{Diag}(U_f,U_{f_1},U_2)$. Since it also leaves $U_{f_1},U_2$ invariant, we deduce that $M M'^{-1}=\mathrm{Id}$ and $M=M'$ proving that $C_fC_f'=C_f'C_f$. Since $C_f,C_f'\in G_f$ are arbitrary,  this shows that   $G_f$ is commutative. This is a contradiction with the assumption that $G_f$ is non-commutative. 
 Consequently, $\widetilde{K}_f=\C(x)$ and this   implies that $f\in \C(x)$. The result is proved.
   
\section{The $G$-functions $\alpha$, $\beta$  and $\xi$}\label{sec:specialGfn}

In this section, we present in details the properties of the function mentioned in Theorem~\ref{theo:N}. We start with the generating series of the sequence of Ap\'ery's numbers, crucial in his proof in \cite{apery} of the irrationality of $\zeta(3)$: 
$$
\alpha (x):=\sum_{n=0}^\infty \left(\sum_{k=0}^n \binom{n}{k}^2\binom{n+k}{n}^2\right)x^n
\in \mathbb Z[[x]].
$$
It is a solution of the differential equation
\begin{multline}\label{eq:diffeqAp}
\quad x^2(1-34x+x^2)y'''(x)+x(3-153x+6x^2)y''(x)\\
+
(1-112x+7x^2)y'(x)+(x-5)y(x)=0. \quad 
\end{multline}
Since $\alpha$ has positive radius of convergence, it is thus a $G$-function. 
Ap\'ery observed in \cite[p.~53]{apery} that Eq.~\eqref{eq:diffeqAp} is the symmetric square of a differential equation of order 2, made explicit by Dwork in \cite[p.~6]{dwork}:
\begin{equation}\label{eq:4bis}
x(x^2-34x+1)y''(x)+(2x^2-51x+1)y'(x)+\frac14(x-10)y(x)=0.
\end{equation}
Moreover, $\alpha(x)=\beta(x)^2$ where $\beta\in \mathbb Q[[x]]$ is a $G$-function and the solution of Eq.~\eqref{eq:4bis} such that $\beta(0)=1$ and $\beta'(0)=5/2$.  Dwork also discussed the importance of this series in the genesis of the famous Bombieri-Dwork conjecture that ``$G$-functions come from geometry'', which he formulated in \cite[p.~2]{dwork}; a differential version of this conjecture was formulated later on by Andr\'e in \cite[p.~111]{andre}.

An expression of $\beta$ in terms of hypergeometric series exists. For this, let us define
$$
\tau(x):=\sum_{n=0}^\infty \left(\sum_{k=0}^n \binom{n}{k}^2\binom{2k}{k}^2\right)x^n.
$$
The main result in \cite{stienstrabeukers} implies that
$
\beta(\frac{x(1-9x)}{1-x})=(1-x)^{1/2}\tau(x)
$
and that
\begin{multline*}
\tau(x)=
\\
\big((3x-1)(3x^3-3x^2+9x-1)\big)^{-1/4} 
{}_{2}F_1 \left[ 
\begin{matrix}
1/12, 5/12
\\
1
\end{matrix};\frac{1728x^6(9x-1)(x-1)^3}{(3x^3-3x^2+9x-1)^3(3x-1)^3}
\right]. 
\end{multline*}
(See also \cite{beukers2} for related results.) 
It follows that $\beta(x)=\mu_0(x)\cdot{}_{2}F_1[1/12,5/12;1;\lambda_0(x)]$ where $\mu_0,\lambda_0 \in\Qbar[[x]]$ are algebraic over $\Qbar(x)$, and $\lambda_0(0)=0$. Hence, neither $\beta$ nor $\alpha$ provide a negative answer to Question~\ref{questionG} and more generally to Question~\ref{questionG'}. But we shall prove that their properties enable us to answer the weaker version of Question~\ref{questionG'} dealt with in Theorem~\ref{theo:N}. \par 

Kovacic's algorithm enables us to compute the differential Galois group of order two differential equations over $\mathbb C(x)$; see \cite{kova} and  \cite{van1999symbolic}. Starting from $\beta$, the first step is to perform a change of functions in order to transform \eqref{eq:4bis} into a differential equation with no term in $y'$. Following \cite[p.~6]{kova}, this ensures that the differential Galois group of the new equation is a subgroup of the unimodular group $\mathrm{SL}_2 (\C)$. 
\begin{prop} \label{prop:2} $(i)$ The function 
$
\nu(x):=x^{1/2}(x^2-34x+1)^{1/4}\beta (x)
$
is transcendental over $\mathbb C(x)$ and solution of the differential equation 
\begin{equation} \label{eq:5}
    4x^2(x^2-34x+1)^2y''(x)+(x^4-44x^3+1206x^2-44x+1)y(x)=0
\end{equation}
with differential Galois group $H_{\nu}$ equal to $\mathrm{SL}_{2}(\C)$. 

\smallskip

$(ii)$ The $G$-function $\xi(x):=\nu^2(x)=x(x^2-34x+1)^{1/2}\alpha (x)$ is transcendental over $\mathbb C(x)$  and the corresponding  differential Galois group $H_{\xi}$ is isomorphic to $\mathrm{PSL}_{2}(\C)$. Moreover, the points $(\sqrt{2}-1)^4$ and $(\sqrt{2}+1)^4$ are non-polar singularities of $\xi$.
\end{prop}

\begin{proof} $(i)$ We split the proof in several steps. 

\medskip

$\bullet$ The transcendence of $\nu$ can be proved in many ways. For instance, Schwarz' classification of algebraic/transcendental hypergeometric series ${}_2F_1[a,b;c;x]$ with rational parameters $a,b,c$ can be applied to the triplet $(a,b,c)=(1/12;5/12,1)$; see \cite[\S VII]{schwarz}.
This is thus also the case of the functions $\beta$ and $\nu$.
\medskip

$\bullet$ It could be checked directly that $\nu$ is solution of Eq.~\eqref{eq:5} starting from the differential equation of order 2 satisfied by $\beta$, but let us explain how it is obtained from Kovacic's method in \cite[p.~5]{kova}. 

We first perform in Eq.~\eqref{eq:4bis} the change of functions $y(x)=\gamma(x) z(x)$, where $\gamma$ will be determined below: we have 
$$
(\gamma'' z+2\gamma' z'+\gamma z'')+\frac{2x^2-51x+1}{x^3-34x^2+x}(\gamma'z+\gamma z')+\frac{1}{4}\frac{x-10}{x^3-34x^2+x}\gamma z=0. 
$$
After reordering this equation in the form $u(x)z''(x)+v(x)z'(x)+w(x)z(x)=0$, we see that in order to have $v(x)=0$, one has to solve the linear differential equation 
$$
\gamma'(x)  =-\frac{(2x^2-51x+1)}{2(x^3-34x^2+x)} \gamma(x).
$$
The general solution is given by $c\cdot x^{-1/2}(x^2-34x+1)^{-1/4}$ where $c\in \mathbb C$ is arbitrary: we now fix the solution $\gamma$ for which $c=1$. Since 
 $$\gamma''=\left(\left(\frac{-(2x^2-51x+1)}{2(x^3-34x^2+x)}\right)'+\left(\frac{-(2x^2-51x+1)}{2(x^3-34x^2+x)}\right)^2\right) \gamma,
 $$
 this enables us to identify the coefficients $u$ and $w$ above. We deduce that 
the function  $\nu:=\beta/\gamma$ is solution of 
$$z''=\left(-\frac{1}{4}\frac{(x-10)}{(x^3-34x^2+x)}+\frac{1}{4}\left(\frac{2x^2-51x+1}{x^3-34x^2+x}\right)^2 +\frac{1}{2}\left(\frac{2x^2-51x+1}{x^3-34x^2+x}\right)'\right)z,
$$
which simplifies to~\eqref{eq:5}. 

\medskip

$\bullet$   We now compute the differential Galois group of the equation~\eqref{eq:5}, that we shall denote by $H_{\nu}$. For this, we use Kovacic's algorithm~\cite{kova}
that computes the differential Galois group of differential equations of order 2 over $\mathbb C(x)$. By \cite[Page 6]{kova}, the differential Galois group $H_{\nu}$ of \eqref{eq:5} lies in $\mathrm{SL}_{2}(\C)$. In \cite[Section 2.1]{kova}, four cases are considered depending on which algebraic subgroup $H_{\nu}\subset \mathrm{SL}_{2}(\C)$ may occur. More precisely:
\begin{enumerate}
\item[--] Case $1$ corresponds to a differential Galois group conjugated to a group of triangular matrices.
\item[--] Case $2$ corresponds to a differential Galois group not in Case $1$, but  conjugated to a subgroup of the  dihedral group.
\item[--] Case $3$ corresponds to a differential Galois group not in Cases $1$ and $2$, and finite.
\item[--] Case $4$ corresponds to the remaining situations, that is when the differential Galois group is $\mathrm{SL}_{2}(\C)$.
\end{enumerate}
We shall prove that we are in Case 4.
Let 
$$
r(x):=-\frac{x^4-44x^3+1206x^2-44x+1}{4x^2(x^2-34x+1)^2}
$$
so that \eqref{eq:5} is simply $y''=ry$. 

\smallskip

-- We first exclude Case $1$. For this, we convert \eqref{eq:5} into an equation where $d/dx$ is replaced by $\delta:=xd/dx$. 
Then $y''=ry$ becomes  $\delta^2 y=\delta y+x^2 r(x)y$, that we write as 
the differential system $\delta Y(x)=A(x)Y(x)$ where $Y={}^t(y,\delta y)$,
\begin{equation}\label{eq:5b}
A(x):= \left(\begin{array}{cc}
0&1\\x^2 r(x) &1
\end{array}\right).
\end{equation}
Since $r(x)=-1/(4x^2)+ \mathcal{O}(x^{-1})$ as $x\to 0$, we have 
$$
A(0)=\left(\begin{array}{cc}
0&1\\ -1/4 &1
\end{array}\right),
$$ 
and this implies that $\delta Y=AY$ has a singularity of the first kind at $x=0$, see \cite[Chapter~2]{balser2008formal} (other authors say that the system is regular singular). 
Let us now compute the Jordan normal form of $A(0)$; its characteristic polynomial is $(\lambda-\frac{1}{2})^2$, so that $\frac{1}{2}$ is the only eigenvalue. Thus, the eigenspace has dimension one because the rank of 
$$
\left(\begin{array}{cc}
-1/2&1\\ -1/4 & 1/2
\end{array}\right)
$$ 
is one. Consequently, the Jordan normal form of $A(0)$ is  
$$\left(\begin{array}{cc}
1/2 &1\\ 0 & 1/2
\end{array}\right)
$$ 
and by \cite[Section~2.1]{balser2008formal} a fundamental matrix solution of \eqref{eq:5b} is given by $$x^{1/2}F(x)\left(\begin{array}{cc}
1&\log(x)\\0 &1
\end{array}\right)\!,
$$ 
where $F\in \mathrm{GL}_2 (\C(\{x\}))$, and $\C(\{x\})$ is the field of germs of meromorphic functions at $0$. 

Then, a basis of solutions of \eqref{eq:5} is given by $\nu(x)$ and $\nu(x)\log(x)+x^{1/2} g(x)$, where $g\in \C(\{x\})$. Assume now that Case $1$ holds. By \cite[Page 9]{kova}, there exists a non-zero solution of \eqref{eq:5} that is solution of a linear differential equation order 1 over $\mathbb C(x)$.  Let 
$$
c_1 \big(\nu (x)\log(x)+x^{1/2}g(x)\big)+c_2 \nu(x)
$$ 
be such a solution with $c_1,c_2\in \C$ not both zero. Recall that $\log (x)$ is transcendental over the field of Puiseux series at $0$. Then, equating the terms in $\log (x)$, we find that necessarily there exists a rational function  $h$ such that  $\nu' (x)=h(x)\nu (x)$. Therefore, there exists a rational function $\tilde{h}$ such that $\beta' (x)=\tilde{h} (x) \beta (x)$, and $\tilde{h}$ is necessarily in $\Qbar(x)$ because $\beta\in \Qbar[[x]]$. Since $G$-functions solutions of order 1 differential equations over $\Qbar(x)$ are algebraic (see for instance \cite[\S2, Proposition~3]{firi2}), we deduce that $\beta$ is algebraic, and $\nu$ as well. This is in contradiction with the first step proved above. Hence, Case 1 does not hold.
  
\medskip  
  
-- As explained in \cite[Section 2.1]{kova}, when we are in Case $2$, then 
$$
\big(\nu(x)\log(x)+x^{1/2}g(x)\big)^2\cdot \big(\nu(x)\big)^2\in \mathbb{C}(x).
$$ 
Using again the transcendence of $\log$ over the field of Puiseux series at $0$, we deduce that this is not possible. Then Case $2$ does not hold.

\medskip

 -- The function $\nu(x)$ is transcendental, so that the differential Galois group can not be finite. Hence,  Case $3$ is also excluded.
 
 \medskip
 
 -- Therefore, we are in Case $4$. This means that $H_{\nu}=\mathrm{SL}_{2}(\C)$. 
 
\bigskip

$(ii)$ The differential Galois group $H_{\nu}=\mathrm{SL}_2 (\C)$  does not satisfy the assumption of the group $G_f$ in Proposition \ref{prop3}. To handle this problem, we replace $\nu$ by its square. 
Then, we need to compute the symmetric square of \eqref{eq:5}.
Recall that $\nu$ is solution of an equation of the form $y''=ry$ with $r\in \C(x)$. Let $y_1,y_2$ be solutions of $y''=ry$ and let us compute a differential equation for  $z:=y_1 y_2$.  We have  $z'=y_1'y_2+y_1y_2'$, and $$z''=2(y_1'y_2')+y_1y_2''+y_1''y_2=2(y_1'y_2')+2r y_1 y_2=2(y_1'y_2')+2rz.$$  
Finally, $$z'''=2y_1'y_2''+2y_1''y_2'+2r'z+2rz'=2ry_1'y_2+2ry_1y_2'+2r'z+2rz'.$$  
Hence the corresponding differential equation  for $z$ is 
\begin{equation}\label{eq7}
z'''=4rz'+2r'z.
\end{equation}
Recall  
that a basis of solutions of \eqref{eq:5} is formed by $\nu(x)$ and  $\mu(x):=\nu(x)\log(x)+x^{1/2}g(x)$. Then a basis of solutions of \eqref{eq7} is formed by  $\nu^{2}$, $2\nu\mu$ and $\mu^{2}$. 

\medskip

$\bullet$ The $G$-function $\xi=\nu^{2}$ is transcendental because $\nu$ is.  Let $\sigma$ be an element of the differential Galois group $H_{\nu}$ of \eqref{eq:5} and let $$
\left(\begin{matrix}
  a   &b  \\
c     & d
\end{matrix} \right)\in \mathrm{SL}_2 (\C)$$ 
be the corresponding matrix. We have $\sigma(\nu)=a\nu+c\mu$ and $\sigma(\mu)=b\nu+d\mu$, so that
\begin{align*}
\sigma (\nu^{2})&=a^2\nu^{2}+ac (2\nu \mu)+c^2\mu^{2},\\
 \sigma (2\nu\mu)   &=   2ab\nu^{2} +(ad+bc) (2\nu\mu)+2cd \mu^{2},\\
 \sigma (\mu^{2}) &=    b^2\nu^{2}+bd (2\nu\mu)+d^2 \mu^{2}.
\end{align*}
Hence, the differential Galois group $H_{\xi}$ of the order three differential equation with a basis of solutions given by $\nu^{2}$, $2\nu\mu$, $\mu^{2}$ is isomorphic to the group of matrices defined by
$$\left(\begin{matrix}
  a^2   &2ab & b^2  \\
ac     & ad+bc & bd \\
c^2& 2cd& d^2
\end{matrix} \right), \quad ad-bc=1.$$
As explained in \cite[p.~13]{singer1993galois}, this group is isomorphic to $\mathrm{PSL}_{2}(\C)$.

\medskip

$\bullet$ Let us prove that $(\sqrt{2}+ 1)^4$ and $(\sqrt{2}- 1)^4$ are non-polar singularities of $\xi$. We shall use Mezzarobba's Maple package {\tt NumGfun} \cite{mezzarobba} for the function $\nu$.~(\footnote{The same commands could be executed directly on the order 3 differential equation satisfied but $\xi$; but the  execution time is much smaller for the order 2 equation satisfied by $\nu$, and the results are then easily transfered to $\xi$ as we do.}) The command {\tt local${}_{-}$basis} 
of this  package shows that a local basis of solutions of \eqref{eq:5} at $x_{\pm}:=(\sqrt{2}\pm 1)^4$ is given by
\begin{align*}
g_{1,\pm}(x)&=(x-x_{\pm})^{1/4} + \Big(6\mp  \frac{271}{64}\sqrt{2}\Big)(x-x_{\pm})^{5/4}+\ldots \in (x-x_{\pm})^{1/4}\Qbar[[x-x_{\pm}]],  
\\
g_{2,\pm}(x)&=
(x-x_{\pm})^{3/4} + \Big(2\mp\frac{271}{192}\sqrt{2}\Big)(x-x_{\pm})^{7/4}+\ldots \in (x-x_{\pm})^{3/4}\Qbar[[x-x_{\pm}]].
\end{align*}
Since $\nu$ is a combination of $g_{1,\star}(x)$ and $g_{2,\star}(x)$ this shows that  $(\sqrt{2}+ 1)^4$ and $(\sqrt{2}- 1)^4$ are non-polar singularities of $\nu$, but this is not sufficient to immediately get the same conclusion for $\xi$. For this, we proceed as follows. 
The analytic continuation of the function $\nu$ is then connected to these bases (locally around $x=x_{\pm}$)  by
$\nu= c_{1, \pm}\cdot g_{1,\pm} + c_{2,\pm} \cdot g_{2,\pm}$,  where the constants $c_{1,\pm},c_{2,\pm}$ can be computed by the command {\tt transition${}_{-}$matrix}:  
\begin{align*}
    c_{1,+}&\approx -0.4827+0.5912i, \quad 
    c_{2,+}\approx 0.1882+0.2304i
   \\
    c_{1,-}&\approx 1.4068+1.4068i, \quad 
    c_{2,-}\approx 0.5484+0.5484i.
\end{align*}
Taking the square, we deduce that the analytic continuation of the function $\xi$ satisfies around $x=x_{\pm}$: 
\begin{multline*}
\xi(x)=c_{1, \pm}^2 \cdot g_{1,\pm}(x)^2 + 2c_{1, \pm} c_{2,\pm}\cdot g_{1,\pm}(x)g_{2,\pm}(x)  + c_{2,\pm}^2\cdot g_{2,\pm}^2(x)\\ 
\in (x-x_{\pm})^{1/2}\Qbar[[x-x_{\pm}]]+\Qbar[[x-x_{\pm}]]
\end{multline*}
where the series in $(x-x_{\pm})^{1/2}\Qbar[[x-x_{\pm}]]$ is non-zero. This proves that $\xi$ has singularities of non-polar type at $x_{+}$ and $x_{-}$.

\medskip

Though this is not needed for our goals, let us also mention that $\xi$  has a singularity at $\infty$ of non-polar type. We proceed a bit differently because, even though a basis of solutions could be computed with {\tt local${}_{-}$basis} by changing $x$ to $1/x$ in Eq.~\eqref{eq:5}, the package does not compute the transition matrix between 0 and $\infty$ directly. Alternatively, one can then use instead the command {\tt analytic${}_{-}$con\-tinuation},  applied to $\nu(x)/x^{1/2}$. This enables us to compute the analytic continuation of $\nu$ around a loop enclosing $\infty$. We get that $\nu$ is transformed to $d_1 \nu+d_2 \omega$, where $\omega(x)=x^{1/2}\log(x)+x^{3/2}(12-6\log(x))+\ldots$ is another local solution at $x=0$ and independent of $\nu$, and $d_1\approx 8034+2229i, d_2\approx -102i$. Hence $\xi$ is transformed to $(d_1 \nu+d_2 \omega)^2$, and we deduce that $\xi$ has a singularity at $\infty$ of non-polar type.
\end{proof}

\section{Proofs of Theorem \ref{theo:N}, Corollary \ref{coro:N} and Corollary \ref{coro:new}}\label{secfin}

\subsection{Proof of Theorem \ref{theo:N}}

Let us fix $N\in \N$.  Let us consider the differential system 
$$
Y'=\left(\begin{matrix}
0&1\\
r&0
\end{matrix}\right) Y=:AY
$$ 
satisfied by the column vector $(\nu,\nu')^\top$, where $r$ is as in the proof of Proposition~\ref{prop:2}. Let $\mathcal{V}$ be a fundamental solution of $Y'=AY$ with $(\nu,\nu')^\top$ as first column. Since  the differential Galois group $G_{\nu}$ of the system is $\mathrm{SL}_2 (\C)$, $\det(\mathcal{V})$ is invariant under the differential Galois group. By the  Galois correspondence (Proposition~\ref{prop1}), $\det(\mathcal{V})\in \C(x)$.
Let $\varphi\in \C(x)$ with $\varphi(0)=0$, and non-constant. Then, $\mathcal{V}(\varphi(x))$ is a fundamental solution of $Y'=\varphi'(x) A(\varphi(x))Y$. Furthermore, $\det(\mathcal{V}(\varphi(x))\in \C(x)$ proving that the corresponding differential Galois group  $G_{\nu\circ\varphi}$  is an algebraic subgroup of $\mathrm{SL}_2 (\C)$ by the Galois correspondence again.  By Lemma~\ref{prop3}, the dimensions of $G_{\nu}$ and $G_{\nu\circ\varphi}$ are equal and since $G_{\nu}=\mathrm{SL}_2 (\C)$, $G_{\nu\circ\varphi}$ has dimension~$3$. Since $\mathrm{SL}_2 (\C)$  is the only algebraic subgroup of $\mathrm{SL}_2 (\C)$ of dimension $3$, we deduce that $G_{\nu\circ\varphi}=\mathrm{SL}_2 (\C)$. Proceeding as in the proof of  Proposition \ref{prop:2}, we find that the differential Galois group of the equation corresponding to $\xi \circ \varphi$
is $\mathrm{PSL}_2 (\C)$.  The latter group satisfies the assumptions on $G_f$ in Proposition \ref{prop3}. Since $\xi$ is transcendental it does not belong to $\C(x)$  and therefore admits at least one non-polar singularity in $\mathbb{P}^1 (\C)$.  Choose  $\varphi_N \in  \mathbb Q(x)$ with $\varphi_N (0)=0$ (in particular  $\varphi_N\in\mathbb \Q[[x]]$), such that the  $G$-function 
$\xi_N=\xi\circ \varphi_N$ has at least $N+1$ non-polar singularities in $\mathbb{P}^1 (\C)$. Furthermore, it is solution of a linear differential equation  with a differential Galois group that satisfies the assumptions on $G_f$ in Proposition~\ref{prop3}.

Recall that $ \mathcal{K}_{N}$ is the field of  rational functions with coefficients in $\overline{\C(x)}$ of solutions of linear differential equations with coefficients in $\mathbb C(x)$ and at most $N$ singularities in $\mathbb{P}^1 (\mathbb C)$. 
To the contrary, assume that  
$\xi_N\in  \mathcal{K}_{N}$. Then it is a rational function with coefficients in $\C(x)$ of solutions of linear differential equations with coefficients in $\mathbb C(x)$ and at most $N$ singularities in $\mathbb{P}^1 (\mathbb C)$ and of algebraic functions.
By Proposition \ref{prop3}, $\xi_N$ belongs either to the Picard-Vessiot extension corresponding to  an algebraic function or to the Picard-Vessiot extension corresponding to a  function solution of a linear differential equation with at most $N$ singularities in $\mathbb{P}^1 (\mathbb C)$. In the first case,  the Picard-Vessiot extension of the algebraic function contains only algebraic functions, and then $\xi_N$ is algebraic.  This contradicts the fact that $\xi_N$ is transcendental.
So the second case holds and any element of the Picard-Vessiot extension is meromorphic at any point distinct from the potential $N$ singularities. Since $\xi_N$ has at least $N+1$ non-polar singularities, we get a contradiction. This concludes the proof of Theorem \ref{theo:N} by proving that $\xi_N\notin  \mathcal{K}_{N}$. 

\subsection{Proof of Corollary \ref{coro:N}}

Let us prove Corollary \ref{coro:N}. Let $M,\widetilde{M}\in \N^*$ and  consider  $\xi_{\widetilde{M}}$ as defined in the proof of Theorem \ref{theo:N}. Assume that $\xi_{\widetilde{M}}\in  \mathcal{F}_{M}$.
Again, $\xi_{\widetilde{M}}$ belongs either to  the Picard-Vessiot extension of an algebraic function (and this is not possible since $\xi_{\widetilde{M}}$ is transcendental),  or to the Picard-Vessiot extension corresponding to a differential equation of the form 
$_p\mathcal{H}_{p-1}[\underline{a};\underline{b}; \lambda]$ with $p \ge 1$, $a_j\in \mathbb Q$,  $b_j\in \mathbb Q\setminus \mathbb Z_{\le 0}$, $\lambda\in {\mathbb C} (x)$ of  numerator and denominator of degree bounded by $M$. The latter differential equation is of order at most $p$, has coefficients in ${\mathbb C}(x)$ and its singularities are amongst the solutions of the three equations $\lambda(x)=\kappa$ where $\kappa\in \{0,1,\infty\}$. Therefore, an element of the Picard-Vessiot extension has at most $3M$ non-polar singularities. Consequently, the transcendental function $\xi_{3M+1}$ does not belong to $\mathcal{F}_{M}$. 

\subsection{Proof of Corollary \ref{coro:new}}

To the contrary, assume that $\xi$ can be written as a polynomial in such functions. 
By Proposition \ref{prop3}, the transcendental function  $\xi$ belongs to the Picard-Vessiot extension of a differential equation of the form  $_p\mathcal{H}_{p-1}[\underline{a};\underline{b}; \lambda]$, where $\lambda(x):=\delta x^m$, $m\neq 0$.
 Now, any singularity $x_0$ of the equation  $_p\mathcal{H}_{p-1}[\underline{a};\underline{b};\lambda]$ must be such that $\delta x_0^m\in \{0,1,\infty\}$.  In particular, all the non-polar singularities of $\xi$ in $\C^*$ must be of modulus $|\delta|^{-1/m}$. But by 
 Proposition~\ref{prop:2}$(ii)$, $\xi$ has two non-polar singularities at $(\sqrt{2}- 1)^4$ and $(\sqrt{2}+ 1)^4$ which are of different modulus. This contradiction concludes the proof.

\bigskip

\noindent Thomas Dreyfus, Institut de Mathématiques de Bourgogne, UMR 5584 CNRS, Université de Bourgogne, F-21000, Dijon, France\\
thomas.dreyfus (at) math.cnrs.fr

\medskip

\noindent Tanguy Rivoal, Institut Fourier, Universit\'e Grenoble Alpes, CNRS, CS 40700, 38058 Grenoble cedex 9, France\\
tanguy.rivoal (at) univ-grenoble-alpes.fr

\bigskip

\noindent Keywords: $G$-functions, Hypergeometric functions, Differential Galois groups, Kovacic's algorithm.

\medskip

\noindent MSC 2020: 33C20, 33E20, 34M15, 34M35.

\end{document}